\documentclass[11pt]{amsart}
\usepackage[dvipsnames]{xcolor}
\usepackage[utf8]{inputenc}
\usepackage{amssymb,amsmath,amsthm,amsfonts,fullpage,float,latexsym,bbm,microtype,cite,fancyvrb}
\usepackage{tikz}
\usepackage{enumerate}
\usepackage{extarrows}
\usetikzlibrary{decorations.pathreplacing}
\usepackage{hyperref}
\hypersetup{colorlinks=true, linkcolor=blue, citecolor=magenta, filecolor=magenta, urlcolor=magenta}
\usepackage{bm,mathrsfs}

\makeatletter
\newtheorem*{rep@theorem}{\rep@title}\newcommand{\newreptheorem}[2]{%
\newenvironment{rep#1}[1]{%
\def\rep@title{\bf #2 \ref{##1}}%
\begin{rep@theorem}}%
{\end{rep@theorem}}}
\makeatother
\newreptheorem{theorem}{Theorem}

\newtheorem{theorem}{Theorem}[section]
\newtheorem{proposition}[theorem]{Proposition}

\newtheorem{conjecture}[theorem]{Conjecture}
\newtheorem{lemma}[theorem]{Lemma}
\newtheorem{corollary}[theorem]{Corollary}
\theoremstyle{definition}
\newtheorem{remark}[theorem]{Remark}

\newtheorem{example}[theorem]{Example}

\renewcommand{\emptyset}{\varnothing}

\newcommand{\C}{\mathbb{C}}

\DeclareMathOperator{\Hilb}{Hilb}

\DeclareMathOperator{\GL}{GL}

\DeclareMathOperator{\Sym}{Sym}

\DeclareMathOperator{\OSP}{OSP}
\DeclareMathOperator{\inv}{inv}
\DeclareMathOperator{\Inv}{Inv}
\DeclareMathOperator{\blocks}{blocks}
\DeclareMathOperator{\sgn}{sgn}

\DeclareMathOperator{\Fub}{Fub}

\DeclareRobustCommand{\qbinom}{\genfrac[]{0pt}{}}

\usepackage{xcolor}
\usepackage[colorinlistoftodos]{todonotes}

\begin{document}

\title{Universal Hilbert series coefficients of the superspace coinvariant ring}

\author[S. Corteel]{Sylvie Corteel}
\address{Department of Mathematics\\
         University of California, Berkeley, CA, USA}
\email{corteel@berkeley.edu}

\author[J. Lentfer]{John Lentfer}
\address{Department of Mathematics\\
         University of California, San Diego, CA, USA}
\email{jlentfer@ucsd.edu}


\begin{abstract}
The coefficients that determine the Hilbert series of the superspace coinvariant ring are indexed by hook-shaped partitions.
We give a manifestly positive combinatorial interpretation of these coefficients, together with several generating functions for them.
Specializing this Hilbert series at $u=-q^2$, we show that its coefficients are differences of binomial coefficients.
Consequently, this proves a conjecture of Sagan--Swanson (2024) that these coefficients are palindromic up to sign.
More generally, for every $m \geq 1$ we obtain closed-form expressions for the $u = -q^m$ specialization.
\end{abstract}

\maketitle

\section{Introduction}\label{sec:introduction}

The superspace coinvariant ring has recently been the subject of much work in algebraic combinatorics. 
The classical coinvariant ring 
\begin{equation}\label{eq:coinv} R_n^{(1,0)} = \C[x_1,\ldots,x_n]/\langle \C[x_1,\ldots,x_n]_+^{\mathfrak{S}_n} \rangle\end{equation}
affords the regular representation of the symmetric group $\mathfrak{S}_n$, and is isomorphic as a graded $\mathfrak{S}_n$-module to the cohomology ring of the flag variety \cite{Borel}.
The superspace coinvariant ring $R_n^{(1,1)}$ extends this setting to include both commutative and anticommutative variables.
Its sign character \cite{SwansonWallach1}, Hilbert series \cite{RhoadesWilson2023}, monomial basis \cite{Angarone2024}, and Frobenius series \cite{MuraiRhoadesWilson} have been determined, resolving conjectures of Sagan--Swanson \cite{SaganSwanson2024} and of Bergeron, Colmenarejo, Li, Machacek, Sulzgruber, and Zabrocki (reported in \cite{MuraiRhoadesWilson}).

The \textbf{superspace coinvariant ring} is defined by
\begin{equation}\label{eq:super-coinv} R_n^{(1,1)} = \C[x_1,\ldots,x_n, \theta_1, \ldots,\theta_n]/\langle \C[x_1,\ldots,x_n, \theta_1, \ldots,\theta_n]_+^{\mathfrak{S}_n} \rangle,\end{equation}
where $\C[x_1,\ldots,x_n, \theta_1, \ldots,\theta_n]_+^{\mathfrak{S}_n}$ denotes those polynomials without constant term which are invariant under the diagonal action of $\mathfrak{S}_n$ permuting each set of variables.
The $x_i$ are commutative variables and the $\theta_i$ are anticommutative variables, so that $x_ix_j = x_jx_i$ and $x_i\theta_j = \theta_j x_i$ but $\theta_i\theta_j = -\theta_j \theta_i$. In particular, $\theta_i^2=0$.
An equivalent construction is $(\Sym \C^n) \otimes (\wedge \C^n) = \C[x_1,\ldots,x_n, \theta_1, \ldots,\theta_n]$.

The superspace coinvariant ring $R_n^{(1,1)}$ is bigraded, where each $x_i$ contributes to the $q$-degree and each $\theta_i$ contributes to the $u$-degree. 
Let $(R_n^{(1,1)})_{r,s}$ denote the subspace of $R_n^{(1,1)}$ which has $q$-degree $r$ and $u$-degree $s$.
Then the bigraded Hilbert series is defined by 
\begin{equation} \Hilb(R_n^{(1,1)}; q; u) := \sum_{r,s \geq 0}\dim \left((R_n^{(1,1)})_{r,s} \right)q^{r} u^{s}.\end{equation}
Rhoades and Wilson proved the following explicit formula for the bigraded Hilbert series of $R_n^{(1,1)}$, where $S[n,k]$ is a $q$-analogue of the Stirling number of the second kind (defined in equation~\eqref{eq:q-Stirling-def}).
\begin{theorem}[\!\!\cite{RhoadesWilson2023}]\label{thm:hilbert}
    For any $n \geq 1$,
    \begin{equation}\label{eq:hilbert-rhoades-wilson}
        \Hilb(R_n^{(1,1)};q;u) = \sum_{k=0}^n [k]_q!S[n,k] u^{n-k}.
    \end{equation}
\end{theorem}
    
We recall diagonal supersymmetry for the Hilbert series of coinvariant rings, in the case of $R_n^{(1,1)}$.
Let $P(k,j,n)$ denote the set of partitions $\lambda$ with $\lambda_{k+1} \leq j$ and $\ell(\lambda) \leq n$.
\begin{proposition}[\!\!\cite{Lentfer-Supersymmetry}]\label{prop:DSS-Hilbert}
    Fix a positive integer $n$. For partitions $\lambda$ with $\ell(\lambda) \leq n$, there exist nonnegative integer coefficients $c_\lambda(n)$ such that the bigraded Hilbert series of $R_n^{(1,1)}$ is
\begin{equation}\label{eq:coeffs-Hilbert}
    \Hilb(R_n^{(1,1)}; q;u) = \sum_{\lambda \in P(1,1,n)} c_\lambda(n) s_\lambda(q/u).
\end{equation}
\end{proposition}

We establish a combinatorial formula for the coefficients $c_\lambda(n)$ defined in equation~\eqref{eq:coeffs-Hilbert}, which determine the super Schur expansion of $\Hilb(R_n^{(1,1)};q;u)$.
\begin{theorem}\label{thm:improved-coefficient-formula}
        For $a \geq 1$ and $b \geq 0$, the coefficients $c_{(a,1^b)}(n)$ from equation~(\ref{eq:coeffs-Hilbert}) may be computed in a manifestly positive manner by
    \begin{equation}\label{eq:second-formula-coeffs}
        c_{(a,1^b)}(n) = \#\{w \in \OSP(n,n-b)\, | \, \inv(w) = a,\ w \text{ is type I} \}.
    \end{equation}
\end{theorem}

\begin{remark}
    Finding a combinatorial interpretation for these coefficients $c_\lambda(n)$, in a slightly different framing, was contributed by the present authors in \cite[Problem 9]{FirstProof} as a challenge for testing the ability of AI systems to autonomously solve short research-level math problems.
    Several AI-generated solutions were proposed and subsequently refereed; some were novel and substantially correct solutions, yet different from our human-generated solution presented here.
\end{remark}

Combining the Rhoades--Wilson formula, diagonal supersymmetry \cite{Lentfer-Supersymmetry}, and a new combinatorial argument on ordered set partitions, we establish another of our main results.
\begin{theorem}\label{thm:main-theorem}
For $n \geq 1$,
    \begin{equation}
    \Hilb(R_n^{(1,1)};q;u)|_{u=-q^2} = 1+ \sum_{i=1}^n q^i \left( \binom{n}{i} - \binom{n}{i-1} \right).
\end{equation}
\end{theorem}

As a consequence, we deduce a conjecture of Sagan--Swanson {\cite[Conjecture 7.9]{SaganSwanson2024}}.
\begin{theorem}\label{thm:sagan-swanson-conjecture}
    For $n \geq 1$, the polynomial 
    \begin{equation}
        \left(\sum_{k=0}^n (-q^2)^{n-k}[k]_q!S[n,k]\right)-1
    \end{equation}
    is palindromic ignoring signs, with the same number of positive and negative coefficients, and where the lower-degree half of nonzero coefficients are positive and the rest are negative.
\end{theorem}

Theorem~\ref{thm:main-theorem} treats the single specialization $u = -q^2$, but algebraic methods apply to $u = -q^m$ for every $m$. 
In Proposition~\ref{prop:H-recurrence}, we give a recurrence relation, and in Theorem~\ref{thm:H-K-closed-form}, we give a new closed-form expression for the Hilbert series specialization $\Hilb(R_n^{(1,1)};q;u)|_{u=-q^m}$.
Consequently, Theorem~\ref{thm:H-K-closed-form} also gives a second, algebraic proof of Theorem~\ref{thm:main-theorem} (as Corollary~\ref{cor:2^n-n-1}).

The specializations $u=-q^m$ are of independent interest. 
As shown by Swanson--Wallach \cite[Section 1]{SwansonWallach2}, each carries homological meaning: $\Hilb(R_n^{(1,1)};q;u)|_{u=-q^m}-1$ is the graded Euler characteristic of a certain generalized exterior derivative cochain complex on the superspace harmonics of $\mathfrak{S}_n$. 
Since the full bigraded Hilbert series is in turn determined by these specializations for $1\leq m\leq n-1$ \cite{SwansonWallach2}, an independent computation of them would give an alternative proof of Theorem~\ref{thm:hilbert}. 
However, we take Theorem~\ref{thm:hilbert} as input to describe the specializations explicitly.

The organization of the paper is as follows. 
We review background in Section~\ref{sec:background}.
We prove Theorem~\ref{thm:improved-coefficient-formula} in Section~\ref{sec:main}.
In Section~\ref{sec:SaganSwanson}, we prove Theorems~\ref{thm:main-theorem} and~\ref{thm:sagan-swanson-conjecture}. In Section~\ref{sec:generating-function-hub}, we turn our attention to the more general setting of the specialization $u=-q^m$, where we prove Proposition~\ref{prop:H-recurrence} and Theorem~\ref{thm:H-K-closed-form}.
In Section~\ref{sec:coefficients}, we study the support and determine the extreme coefficients arising from the specialization $u=-q^m$.
In Section~\ref{sec:row-gf}, we study the generating function $\sum_a c_{(a,1^b)}(n) q^a$, giving a closed-form solution in Theorem~\ref{thm:unsigned-gf}.

\section{Background}\label{sec:background}

In this section, we recall essential background on super Schur functions, $q$-Stirling numbers, and the combinatorics of ordered set partitions. 
We assume the reader has some basic familiarity with symmetric functions (see for example \cite{Macdonald} or \cite[Chapter 7]{StanleyEC2}). 

A partition $\lambda = (\lambda_1, \ldots, \lambda_{\ell(\lambda)})$ is a weakly decreasing sequence of positive integers $\lambda_1 \geq \cdots \geq \lambda_{\ell(\lambda)} > 0$; its length is $\ell(\lambda)$.
We say $\lambda$ is a partition of $n$, denoted by $\lambda \vdash n$, if $\lambda_1 + \cdots + \lambda_{\ell(\lambda)} = n$.
A partition $\lambda$ may be identified with its Ferrers diagram, which is a left-justified arrangement of $\lambda_i$ boxes in each row $i$ from $1$ to $\ell(\lambda)$.
If $\mu \subseteq \lambda$, meaning the Ferrers diagram of $\mu$ lies inside that of $\lambda$, the skew diagram $\lambda/\mu$ is obtained by deleting the boxes of $\mu$ from those of $\lambda$.
Denote a Schur function by $s_\lambda(q_1,\ldots,q_k)$ and a skew Schur function by $s_{\lambda/\mu}(q_1,\ldots,q_k)$.
A \textbf{super Schur function} $s_\lambda(q_1,\ldots,q_k/u_1,\ldots,u_j)$ is defined by 
\begin{equation}
    s_\lambda(q_1,\ldots,q_k/u_1,\ldots,u_j) = \sum_{\nu \subseteq \lambda} s_\nu(q_1,\ldots,q_k)s_{\lambda'/\nu'}(u_1,\ldots,u_j),
\end{equation}
where $\lambda'$ denotes the transpose of $\lambda$ (\!\!\cite{BereleRegev}; see also \cite[Sections A.2.2 and 3.2]{ChengWangBook}).

A partition $\lambda$ with $\lambda_2 \leq 1$ is called a \textbf{hook shape}. 
Note that $P(1,1,n)$ consists of all hook shape partitions of length at most $n$.
The super Schur function has a simple formula for hook shape partitions.

\begin{lemma}\label{lem:schur-1-1}
    If $\lambda$ is $(a,1^b)$ for some $a \geq 1$ and $b \geq 0$, then $s_\lambda(q/u) = q^au^b + q^{a-1}u^{b+1}$. 
    If $\lambda$ is $\varnothing$, then $s_\lambda(q/u) = 1$.
    If $\lambda$ is not contained in a hook shape, then $s_\lambda(q/u) = 0$.
\end{lemma}

By taking the specialization $u=-q^m$, we obtain the following.

\begin{corollary}\label{cor:schur_specialization_-q^m=u}
    If $\lambda$ is $(a,1^b)$ for some $a \geq 1$ and $b \geq 0$, then
    \begin{equation}
        s_\lambda(q/u)|_{u=-q^m} = (-1)^bq^{a+mb}(1-q^{m-1}).
    \end{equation}
    If $\lambda$ is $\varnothing$, then $s_\lambda(q/u)|_{u=-q^m} = 1$. If $\lambda$ is not contained in a hook shape, then $s_\lambda(q/u)|_{u=-q^m} = 0$.
\end{corollary}

Let $[k]_q$ denote the $q$-integer $1+q+\cdots+q^{k-1}$. Define the $q$-factorial $[k]_q!$ by $[k]_q [k-1]_q \cdots [2]_q [1]_q$, and let $[0]_q! := 1$. 
For $0 \leq k \leq n$, define the $q$-binomial coefficients by 
\begin{equation}
    \qbinom{n}{k}_q = \frac{[n]_q!}{[k]_q! [n-k]_q!}.
\end{equation}

An \textbf{ordered set partition of $\{1,\ldots,n\}$ with $k$ blocks} is a sequence of nonempty sets $w = (S_1/ S_2/ \cdots / S_k)$ such that the disjoint union of the sets $S_i$ is $\{1,\ldots,n\}$. 
We denote the set of all ordered set partitions of $\{1,\ldots,n\}$ into $k$ blocks by $\OSP(n,k)$, and the set of all ordered set partitions of $\{1,\ldots,n\}$ by $\OSP(n)$.
Ordered set partitions were studied by Haglund, Rhoades, and Shimozono in the context of generalized coinvariant algebras \cite{HaglundRhoadesShimozono} related to the Delta conjecture (conjectured in \cite{HaglundRemmelWilson2018} and proven in \cite{DAdderioMellit}).
Subsequently, Sagan--Swanson connected them to the superspace coinvariant ring. 
Following Sagan--Swanson \cite[Section 5.1]{SaganSwanson2024}, define the \textbf{inversion set} on ordered set partitions by
\begin{equation}
    \Inv(w) = \{(s,S_j)\, | \, s \in S_i \text{ for some } i < j \text{ and } s \geq \min S_j\},
\end{equation}
and the \textbf{inversion statistic} by $\inv(w) = \# \Inv(w)$.

Recall the \textbf{$q$-Stirling numbers of the second kind} $S[n,k]$, which are defined by the initial conditions $S[0,k] = \delta_{0,k}$ and $S[n,0] = \delta_{n,0}$, along with the recurrence relation
\begin{equation}\label{eq:q-Stirling-def}
    S[n,k] = S[n-1,k-1] + [k]_q S[n-1,k]
\end{equation}
for $1 \leq k \leq n$.

The \textbf{$q$-ordered Stirling numbers of the second kind} $S^o[n,k]$ are defined by $[k]_q!S[n,k]$. 
They satisfy the recurrence relation
\begin{equation}
    S^o[n,k] = [k]_q \left(S^o[n-1,k-1] + S^o[n-1,k]\right),
\end{equation}
for $1 \leq k \leq n$, with initial conditions $S^o[n,0]=\delta_{n,0}$ and $S^o[0,k]=\delta_{k,0}$, and the convention $S^o[n,k] := 0$ for $k > n$.

We will use the following expression for $S^o[n,k]$.
\begin{lemma}\label{lem:SO-closed-form} 
For $0 \leq k \leq n$,
    \begin{equation}
S^o[n,k]= q^{-\binom{k}{2}}\sum_{j=0}^k (-1)^j q^{\binom{j}{2}} \qbinom{k}{j}_q [k-j]_q^n.
\end{equation}
\end{lemma}

\begin{proof}
    Carlitz's identity (originally \cite{Carlitz1948}; see \cite[Corollary 5.2]{YueEhrenborgReaddy}) says $[r]_q^n = \sum_{k=0}^n \qbinom{r}{k}_q q^{\binom{k}{2}} S^o[n,k]$.
    The result is immediate upon applying $q$-binomial (Gauss) inversion \cite[Theorem 1.2]{Ernst}.
\end{proof}

We recall the following generating function for $q$-ordered Stirling numbers of the second kind, in terms of the inversion statistic.
\begin{proposition}[\!\!{\cite[Theorem 5.1]{SaganSwanson2024}}]\label{prop:OSP-inv}
For $0 \leq k \leq n$,
    \begin{equation}
        [k]_q! S[n,k] = \sum_{w \in \OSP(n,k)}q^{\inv(w)}.
    \end{equation}
\end{proposition}
Let $\langle q^i\rangle f(q)$ denote the coefficient of $q^i$ in a power series $f(q)$. 
Coefficient-wise, the proposition says
\begin{equation}\label{eq:OSP-equation}
    \langle q^i \rangle [k]_q! S[n,k] = \#\{w \in \OSP(n,k)\, | \, \inv(w) = i \}.
\end{equation}

Now we recall the \textbf{involution $\phi$ of Sagan--Swanson} \cite[Section 5.1]{SaganSwanson2024}.
Given an ordered set partition $w=(S_1/S_2/\ldots /S_k)$, we call $w$ \textbf{splittable} at $M = \max S_i$, for a fixed $i \in\{1,\ldots,k\}$, if $|S_i|>1$; we call $w$ \textbf{mergeable} at $M= \max S_i$, for a fixed $i \in\{1,\ldots,k-1\}$, if $|S_i|=1$ and $M>\max S_{i+1}$.

Define a function $\phi$ on ordered set partitions of $\{1,\ldots,n\}$.
Given $w=(S_1/S_2/\ldots /S_k)$, find the largest $M$ such that $w$ is either splittable or mergeable at $M$, if such an $M$ exists.
We say that
\begin{itemize}
    \item $w$ is \textbf{type I} if $w$ is mergeable at $M$;
    \item $w$ is \textbf{type II} if $w$ is splittable at $M$;
    \item $w$ is \textbf{type III} if no $M$ exists.
\end{itemize}
Then 
\begin{equation}\label{eq:phi-mergeable-splittable}
\phi(w)=\begin{cases}
   (S_1/S_2/\ldots /S_{i-1}/S_i\cup S_{i+1}/S_{i+2}/\ldots  /S_k) & \text{if $w$ is type I}, \\
   (S_1/S_2/\ldots /S_{i-1}/\{M\}/S_i-\{M\}/S_{i+1}/\ldots  /S_k) & \text{if $w$ is type II}, \\
   w & \text{if $w$ is type III}.
\end{cases}
\end{equation}
It follows that $\phi$ is an involution \cite[Section 5]{SaganSwanson2024}.
Let $\blocks(w)$ denote the number of blocks of $w$.
As merging removes one inversion and one block while splitting adds one of each, the quantity $n-\blocks(w)+\inv(w)$ is preserved by $\phi$.
Finally, with the definition $\sgn(w) = (-1)^{n- \blocks(w)}$, $\phi$ is a sign-reversing involution on all non-fixed points.

\section{Formulas for the coefficients \texorpdfstring{$c_{(a,1^b)}(n)$}{c(a,1b)(n)}}\label{sec:main}

In this section, we work towards finding a manifestly positive combinatorial formula for the coefficients $c_{(a,1^b)}(n)$ in order to establish Theorem~\ref{thm:improved-coefficient-formula}. 
First, we express the coefficients in terms of a signed enumeration of certain ordered set partitions.

\begin{lemma}\label{lem:c_{(a,1^b)}}
    For $a \geq 1$ and $b \geq 0$, the coefficients $c_{(a,1^b)}(n)$ from equation~(\ref{eq:coeffs-Hilbert}) may be computed by
    \begin{equation}\label{eq:first-formula-coeffs}
        c_{(a,1^b)}(n) = \sum_{i=0}^b (-1)^i \#\{w \in \OSP(n,n-b+i)\, | \, \inv(w) = a+i \}.
    \end{equation}
\end{lemma}

\begin{proof}
By equations~(\ref{eq:coeffs-Hilbert}) and~(\ref{eq:hilbert-rhoades-wilson}), we have that
\begin{equation}
    \sum_{\lambda \in P(1,1,n)} c_\lambda(n) s_\lambda(q/u) = \sum_{d=0}^n [d]_q!S[n,d] u^{n-d}.
\end{equation}
Then by Lemma~\ref{lem:schur-1-1}, we have
\begin{equation}
    1 + \sum_{\substack{a \geq 1,\\ b \geq 0}} c_{(a,1^b)}(n)(q^au^b+q^{a-1}u^{b+1}) = \sum_{d=0}^n [d]_q!S[n,d] u^{n-d}.
\end{equation}
By equation~(\ref{eq:OSP-equation}),
\begin{equation}
    1 + \sum_{\substack{a \geq 1,\\ b \geq 0}} c_{(a,1^b)}(n)(q^au^b+q^{a-1}u^{b+1}) = \sum_{d=0}^n  \sum_{i \geq 0} q^i \# \{w \in \OSP(n,d)\, |\, \inv(w) = i \} u^{n-d}.
\end{equation}
By extracting the coefficient of $q^a u^b$ from both sides, we find that, for $b \geq 1$,
\begin{equation}
    c_{(a,1^b)}(n)+c_{(a+1,1^{b-1})}(n) = \#\{w\in \OSP(n,n-b)\ | \ \inv(w)=a\}
\end{equation}
and for $b=0$,
\begin{equation}
    c_{(a)}(n) = \#\{w\in \OSP(n,n)\ | \ \inv(w)=a\}.
\end{equation}
Finally, we solve for $c_{(a,1^b)}(n)$ to establish equation~(\ref{eq:first-formula-coeffs}).
\end{proof}

We wish to improve upon Lemma~\ref{lem:c_{(a,1^b)}} to get a manifestly positive formula for the coefficients $c_{(a,1^b)}(n)$. 
To do so, first we will review the bijection between ordered set partitions and certain pairs of subsets and inversion sequences (which themselves correspond to monomials in the super-Artin basis for $R_n^{(1,1)}$ \cite{Angarone2024}).

For $0 \leq k \leq n-1$, define $\mathscr{B}(n,k)$ to be the set of all $k$-element subsets of $\{2,3,\ldots,n\}$.
To a $B \in \mathscr{B}(n,k)$, we may associate any inversion sequence $A = (a_1,\ldots, a_n)$, where the $a_i$ satisfy $0 \leq a_i < i -|B \cap \{1,\ldots,i\}|$.

There is a bijection between $\OSP(n,n-k)$ and all pairs $(B,A)$ where $B \in \mathscr{B}(n,k)$ and $A$ is an inversion sequence for $B$. 
We describe the bijection starting from a pair $(B,A)$ (by specializing the map from \cite[Section 6]{Lentfer2024}).
Initialize an ordered set partition as ``$\{1\}$.'' 
For $i$ ranging from $2$ to $n$, do the following:
\begin{enumerate}[(a)]
    \item If $i \not\in B$, insert ``$/\{i\}$'' or ``$\{i\}/$'' in such a way as to create a new block consisting of the set $\{i\}$ at block position $a_i$ from the right (indexing starting at 0);
    \item If $i \in B$, take the union of $\{i\}$ with the existing block at block position $a_i$ from the right.
\end{enumerate}
After this has been performed up through $i=n$, the output is an ordered set partition $w$.
Identify the two forms $(B,A)$ and $w$ under this bijection.
In particular, $B$ has cardinality $n-\blocks(w)$ and $A$ satisfies $\sum_{i=1}^n a_i = \inv(w)$.

\begin{example}
    We give an example of the bijection between pairs $(B,A)$ and ordered set partitions at $n=3$.
    \begin{itemize}
    \item Consider $\mathscr{B}(3,0) =\big\{\varnothing\big\}$. 
    Associated to $B = \varnothing$, there are inversion sequences $A = (0,0,0)$, $(0,1,0)$, $(0,0,1)$, $(0,1,1)$, $(0,0,2)$, and $(0,1,2)$. 
    We follow the procedure to create ordered set partitions.
    \begin{itemize}
        \item For $A = (0,0,0)$, we build $\{1\}$, then $\{1\}/\{2\}$, then finally $\{1\}/\{2\}/\{3\}$.
        \item For $A = (0,1,0)$, we build $\{1\}$, then $\{2\}/\{1\}$, then finally $\{2\}/\{1\}/\{3\}$.
        \item For $A = (0,0,1)$, we build $\{1\}$, then $\{1\}/\{2\}$, then finally $\{1\}/\{3\}/\{2\}$.
        \item For $A = (0,1,1)$, we build $\{1\}$, then $\{2\}/\{1\}$, then finally $\{2\}/\{3\}/\{1\}$.
        \item For $A = (0,0,2)$, we build $\{1\}$, then $\{1\}/\{2\}$, then finally $\{3\}/\{1\}/\{2\}$.
        \item For $A = (0,1,2)$, we build $\{1\}$, then $\{2\}/\{1\}$, then finally $\{3\}/\{2\}/\{1\}$.
    \end{itemize}
    \item Consider $\mathscr{B}(3,1) =\big\{\{2\},\{3\}\big\}$.
    Associated to $B = \{2\}$, there are inversion sequences $A = (0,0,0)$ and $(0,0,1)$.  
    We follow the procedure to create ordered set partitions.
    \begin{itemize}
        \item For $A = (0,0,0)$, we build $\{1\}$, then $\{1,2\}$, then finally $\{1,2\}/\{3\}$.
        \item For $A = (0,0,1)$, we build $\{1\}$, then $\{1,2\}$, then finally $\{3\}/\{1,2\}$.
    \end{itemize}
    Associated to $B = \{3\}$, there are inversion sequences $A = (0,0,0)$, $(0,1,0)$, $(0,0,1)$, and $(0,1,1)$.
    We follow the procedure to create ordered set partitions.
    \begin{itemize}
        \item For $A = (0,0,0)$, we build $\{1\}$, then $\{1\}/\{2\}$, then finally $\{1\}/\{2,3\}$.
        \item For $A = (0,1,0)$, we build $\{1\}$, then $\{2\}/\{1\}$, then finally $\{2\}/\{1,3\}$.
        \item For $A = (0,0,1)$, we build $\{1\}$, then $\{1\}/\{2\}$, then finally $\{1,3\}/\{2\}$.
        \item For $A = (0,1,1)$, we build $\{1\}$, then $\{2\}/\{1\}$, then finally $\{2,3\}/\{1\}$.
    \end{itemize}
    \item Consider $\mathscr{B}(3,2) =\big\{\{2,3\}\big\}$.
    Associated to $B = \{2,3\}$, there is one inversion sequence $A = (0,0,0)$.
    We follow the procedure to create an ordered set partition.
    \begin{itemize}
        \item For $A = (0,0,0)$, we build $\{1\}$, then $\{1,2\}$, then finally $\{1,2,3\}$.
    \end{itemize}
    \end{itemize}
\end{example}

We explain the involution $\phi$ of Sagan--Swanson on the ordered set partitions under the aforementioned bijection.
Given an ordered set partition $w=(B,A)$ in $\OSP(n)$, we say that $w$ is of 
\begin{itemize}
    \item \textbf{type I} if there exists an index $j$ such that $a_j>0$ and $a_i=0$ for all $i>j$ and $i\notin B$ for all $i\ge j$;
    \item \textbf{type II} if there exists an index $j$ such that $j\in B$ and $a_i=0$ for all $i>j$ and $i\notin B$ for all $i>j$;
    \item \textbf{type III} if no such $j$ exists (i.e., if $a_i=0$ for $1\le i\le n$ and $B=\emptyset$).
\end{itemize}
Once such a $j$ is determined, the involution $\phi$ is then given by: 
\begin{itemize}
\item If $w$ is of type I then $\phi(w)=(\tilde B,\tilde A)$ where $\tilde B=B\cup\{j\}$ and $\tilde A$ is given by $\tilde a_i=a_i$ for all $i\neq j$ and $\tilde a_j=a_j-1$. Then $\phi(w)$ is of type II.
\item If $w$ is of type II then $\phi(w)=(\tilde B,\tilde A)$ where $\tilde B=B\backslash\{j\}$ and $\tilde A$ is given by $\tilde a_i=a_i$ for all $i\neq j$ and $\tilde a_j=a_j+1$. Then $\phi(w)$ is of type I.
\item If $w$ is of type III then $\phi(w)=w$ and it is a fixed point.
\end{itemize}

\begin{example}
    We classify all 13 ordered set partitions at $n=3$ as type I, II, or III and show the effect of the involution $\phi$.
    The following six pairs have the type I ordered set partition on the left and the type II ordered set partition on the right:
    \begin{itemize}
        \item $\big(\varnothing, (0,1,0)\big) \xlongleftrightarrow[]{\phi} \big(\{2\}, (0,0,0)\big)$.
        \item $\big(\varnothing, (0,0,1)\big) \xlongleftrightarrow[]{\phi} \big(\{3\}, (0,0,0)\big)$.
        \item $\big(\varnothing, (0,1,1)\big) \xlongleftrightarrow[]{\phi} \big(\{3\}, (0,1,0)\big)$.
        \item $\big(\varnothing, (0,0,2)\big) \xlongleftrightarrow[]{\phi} \big(\{3\}, (0,0,1)\big)$.
        \item $\big(\varnothing, (0,1,2)\big) \xlongleftrightarrow[]{\phi} \big(\{3\}, (0,1,1)\big)$.     
        \item $\big(\{2\}, (0,0,1)\big) \xlongleftrightarrow[]{\phi} \big(\{2,3\}, (0,0,0)\big)$.        
    \end{itemize}
    There is one type III ordered set partition:
    \begin{itemize}
        \item $\big(\varnothing, (0,0,0)\big)$.
    \end{itemize}
\end{example}

For $\ell \geq 1$ and $0 \leq b \leq {\lfloor \frac{\ell-1}{2}\rfloor}$, define the set
\begin{equation}
\mathcal S_{n,\ell,b}=\bigcup_{c=0}^b  \{w\in\OSP(n,n-c)\, |\, \inv(w)=\ell-b-c\}.
\end{equation}
While $\phi$ is an involution on $\OSP(n)$, it does not generally fix $\mathcal S_{n,\ell,b}$. 
We describe when points of $\mathcal S_{n,\ell,b}$ remain or leave the set under $\phi$ in the following lemma.
\begin{lemma}\label{lem:pairing-lemma}
    If $w \in \mathcal S_{n,\ell,b}$, then:
\begin{enumerate}[(i)]
    \item if $w$ is of type I, $w \in \OSP(n,n-c)$ and $\inv(w) = \ell-2c$, then $\phi(w) \not\in \mathcal S_{n,\ell,b}$;
    \item if $w$ is otherwise of type I, then $\phi(w) \in \mathcal S_{n,\ell,b}$;
    \item if $w$ is of type II, then $\phi(w) \in \mathcal S_{n,\ell,b}$;
    \item $w$ cannot be of type III.
\end{enumerate}
\end{lemma}

\begin{proof}
Let $b$ be fixed. The condition on $\inv(w)$ for $\mathcal S_{n,\ell,b}$ requires that $\ell - 2b \leq \inv(w) \leq \ell-b$, by considering the values $c$ can take on.
\begin{enumerate}[(i)]
    \item If $w$ is of type I, $w \in \OSP(n,n-c)$ and $\inv(w) = \ell-2c$, this means that $c=b$, so $\inv(w) = \ell-2b$. As $\phi$ applied to an ordered set partition of type I lowers the number of inversions by 1, $\inv(\phi(w)) = \ell-2b-1$, so $\phi(w) \not \in \mathcal S_{n,\ell,b}$.
    \item If $w$ is otherwise of type I, then $c \leq b-1$. Since here $\phi$ removes one block and one inversion, we get that $\inv(\phi(w))= \ell - b - c -1 \geq \ell-2b$. Thus $\phi(w) \in \mathcal S_{n,\ell,b}$.
    \item If $w$ is of type II, then $\blocks(w) \leq n-1$, so $c \geq 1$, which implies that $\inv(w) \leq \ell-b-1$. 
    Since here $\phi$ adds one block and one inversion, $\inv(\phi(w)) \leq \ell - b$, so $\phi(w) \in \mathcal S_{n,\ell,b}$.
    \item Suppose towards a contradiction that  $w \in \mathcal S_{n,\ell,b}$ is of type III, then $c=0$ and $\ell = b$. Since $\ell \geq 1$, then $b > {\lfloor \frac{\ell-1}{2}\rfloor}$, contradicting that $b \leq {\lfloor \frac{\ell-1}{2}\rfloor}$.
\end{enumerate}
\end{proof}

Let us collect those ordered set partitions in $\mathcal S_{n,\ell,b}$ which are not paired under $\phi$ as 
\begin{equation}
    \mathcal T_{n,\ell,b}=
\{w\in \OSP(n,n-b) \, |\, \text{$w$ of type I and}\ \inv(w)=\ell-2b\},
\end{equation}
and define
\begin{equation}
    \mathcal T_{n,\ell} = \bigcup_{b=0}^{\lfloor\frac{\ell-1}{2}\rfloor} \mathcal T_{n,\ell,b}.
\end{equation}

Now we are ready to prove our combinatorial interpretation of the coefficients $c_{(a,1^b)}(n)$.

\begin{proof}[Proof of Theorem~\ref{thm:improved-coefficient-formula}]
By Lemma~\ref{lem:c_{(a,1^b)}}, we wish to enumerate the points of 
\begin{equation}
\mathcal S_{n,a+2b,b} = \bigcup_{i=0}^b \{ w\in \OSP(n,n-b+i)\, | \, \inv(w) = a+i\}
\end{equation}
which are not paired under the involution $\phi$ of Sagan--Swanson.
This is because the points which are paired under $\phi$ have opposite signs and cancel in the signed sum of Lemma~\ref{lem:c_{(a,1^b)}}.
By Lemma~\ref{lem:pairing-lemma}, those points in $\mathcal S_{n,a+2b,b}$ which are not paired under $\phi$ are
\begin{equation}
    \mathcal T_{n,a+2b,b} = \{w \in \OSP(n,n-b)\, | \, \inv(w) = a, \ w \text{ is type I} \}.
\end{equation}
Since each unpaired point lies in $\OSP(n,n-b)$, it is counted in the $i=0$ term of equation~\eqref{eq:first-formula-coeffs}, with sign $+1$. 
This completes the proof.
\end{proof}

The combinatorial formula of Theorem~\ref{thm:improved-coefficient-formula} lets us describe exactly which coefficients $c_{(a,1^b)}(n)$ are nonzero.
Regarding the empty partition, we have $c_\varnothing(n) = 1$.

\begin{proposition}\label{prop:nonzero-support}
    Fix $n \geq 1$, and let $a \geq 1$ and $b \geq 0$.
    Then $c_{(a,1^b)}(n) > 0$ if and only if
    \begin{equation}\label{eq:support}
        0 \leq b \leq n-2 \quad \text{and} \quad 1 \leq a \leq \binom{n}{2} - \binom{b+1}{2} - b.
    \end{equation}
\end{proposition}

\begin{proof}
    By Theorem~\ref{thm:improved-coefficient-formula},
    \begin{equation}
        c_{(a,1^b)}(n) = \#\{w \in \OSP(n,n-b) \ | \ \inv(w) = a,\ w \text{ is type I}\},
    \end{equation}
    so $c_{(a,1^b)}(n) > 0$ if and only if there exists a type I ordered set partition $w \in \OSP(n,n-b)$ with $\inv(w) = a$.
    We use the encoding $w = (B,A)$ where $|B| = b$, $\inv(w) = \sum_{i=1}^n a_i$, and the entries $a_i$ of the inversion sequence satisfy $0 \leq a_i \leq \kappa_i$, where
    \begin{equation}\label{eq:kappa}
        \kappa_i := i - 1 - |B \cap \{1,\ldots,i\}|.
    \end{equation}
    Recall that $w$ is type I exactly when there exists an index $j$ with $a_j > 0$, with $a_i = 0$ for all $i > j$, and with $i \notin B$ for all $i \geq j$.

    Suppose $w = (B,A)$ is type I, with associated index $j$.
    Since $i \notin B$ for all $i \geq j$, we have $B \subseteq \{2,\ldots,{j-1}\}$, so  $|B \cap \{1,\ldots,i\}| = b$ for every $i \geq j$. 
    That is, $\kappa_j = j - 1- b$.
    As $a_j \geq 1$ implies that $\kappa_j \geq 1$, we obtain $j \geq b+2$, and since $j \leq n$ this implies $b \leq n-2$.

    By the previous paragraph, $n \notin B$, so $B \subseteq \{2,\ldots,n-1\}$.
    Moving an element of $B$ to a larger index weakly increases every $\kappa_i$, hence the bound $\inv(w) = \sum_{i=1}^n a_i \leq \sum_{i=1}^n \kappa_i$, so this bound is maximized by taking $B$ to be the $b$ largest elements of $\{2,\ldots,n-1\}$, namely $B = \{n-b,\ldots,n-1\}$.
    In this case, we compute by equation~\eqref{eq:kappa} that
    \begin{equation}
        \kappa_i = \begin{cases}
            i-1 & \text{ if } 1 \leq i \leq n-b-1,\\
            n-b-2& \text{ if } n-b \leq i \leq n-1,\\
            n-b-1& \text{ if } i=n.
        \end{cases}
    \end{equation}
    Then $a$ is bounded above by
    \begin{equation}\label{eq:max}
         \sum_{i=1}^n \kappa_i = \binom{n-b-1}{2} + b(n-b-2) + (n-b-1) = \binom{n}{2} - \binom{b+1}{2} - b.
    \end{equation}
    This establishes that if $c_{(a,1^b)}(n) > 0$, then $a$ and $b$ satisfy equation~\eqref{eq:support}. 
    
    Next, we show that the conditions in equation~\eqref{eq:support} are sufficient.
    For fixed $b$ with $0 \leq b \leq n-2$, again consider the set $B = \{n-b,n-b+1,\ldots,n-1\}$.
    Any inversion sequence $A$ with $a_n \geq 1$ yields a type I ordered set partition $(B,A)$:
    the index $j=n$ is our witness for being type I, since $a_n > 0$, no index exceeds $n$, and $n \notin B$.
    Letting each $a_i$ range over $\{0,1,\ldots, \kappa_i\}$ for $1 \leq i \leq n-1$ and $a_n$ range over $\{1,\ldots,\kappa_n\}$, by equation~\eqref{eq:max}, the sum $\sum_{i=1}^n a_i$ could take on any integer value in $\{1,\ldots, \binom{n}{2} - \binom{b+1}{2} - b\}$.
    Hence for each such $a$, there exists a type I ordered set partition $w \in \OSP(n,n-b)$ with $\inv(w)=a$. 
    Hence $c_{(a,1^b)}(n) > 0$.
\end{proof}

\begin{example}
    We compute all coefficients $c_{(a,1^b)}(n)$ at $n=3$. From the previous examples, we have six type I ordered set partitions. We compute their inversion statistics:
    \begin{itemize}
        \item $\big(\varnothing, (0,1,0)\big) = \{2\}/\{1\}/\{3\}$ has $1$ inversion.
        \item $\big(\varnothing, (0,0,1)\big) = \{1\}/\{3\}/\{2\}$ has $1$ inversion.
        \item $\big(\varnothing, (0,1,1)\big) = \{2\}/\{3\}/\{1\}$ has $2$ inversions.
        \item $\big(\varnothing, (0,0,2)\big) = \{3\}/\{1\}/\{2\}$ has $2$ inversions.
        \item $\big(\varnothing, (0,1,2)\big) = \{3\}/\{2\}/\{1\}$ has $3$ inversions. 
        \item $\big(\{2\}, (0,0,1)\big) = \{3\}/\{1,2\}$ has $1$ inversion.   
    \end{itemize}
    Thus we compute $c_{(1)}(3) = 2$, $c_{(2)}(3) = 2$, $c_{(3)}(3) = 1$, and $c_{(1,1)}(3) = 1$. We always have $c_\varnothing(3) = 1$. Then all other $c_{(a,1^b)}(3) = 0$. 
\end{example}

\begin{example}
We compute all coefficients $c_{(a,1^b)}(n)$ at $n=4$. For brevity, we drop the set braces in the ordered set partitions.
\begin{figure}[h]
\begin{tabular}{l|l|llllll}
\text{Coefficient} & \text{Value} & \text{OSPs}         \\ \hline
$c_{(1)}(4)$       & 3            & $2|1|3|4$  & $1|3|2|4$ & $1|2|4|3$ &           &           &           \\
$c_{(1,1)}(4)$     & 3            & $1|4|23$   & $12|4|3$  & $3|12|4$  &           &           &           \\
$c_{(1,1,1)}(4)$   & 1            & $4|123$    &           &           &           &           &           \\
$c_{(2)}(4)$       & 5            & $2|1|4|3$  & $1|3|4|2$ & $2|3|1|4$ & $3|1|2|4$ & $1|4|2|3$ &           \\
$c_{(2,1)}(4)$     & 5            & $4|1|23$   & $13|4|2$  & $2|4|13$  & $3|4|12$  & $4|12|3$  &           \\
$c_{(3)}(4)$       & 6            & $2|3|4|1$  & $3|2|1|4$ & $2|4|1|3$ & $3|1|4|2$ & $4|1|2|3$ & $1|4|3|2$ \\
$c_{(3,1)}(4)$     & 4            & $4|13|2$   & $23|4|1$  & $4|2|13$  & $4|3|12$  &           &           \\
$c_{(4)}(4)$       & 5            & $4|1|3|2$  & $4|2|1|3$ & $3|4|1|2$ & $3|2|4|1$ & $2|4|3|1$ &           \\
$c_{(4,1)}(4)$     & 1            & $4|23|1$   &           &           &           &           &           \\
$c_{(5)}(4)$       & 3            & $4|2|3|1$  & $4|3|1|2$ & $3|4|2|1$ &           &           &           \\
$c_{(6)}(4)$       & 1            & $4|3|2|1$  &           &           &           &           &          
\end{tabular}
\caption{The coefficients $c_\lambda(4)$ and their contributing ordered set partitions (OSPs) given by Theorem~\ref{thm:improved-coefficient-formula}. Aside from $c_\varnothing(4)=1$, all coefficients $c_\lambda(4)$ not shown are $0$.}
\end{figure}

One can compute that
\begin{align*}
    \Hilb(R_4^{(1,1)};q;u) &=  (q^6+3q^5+5q^4+6q^3+5q^2+3q+1)+ (q^5+4q^4+9q^3+11q^2+8q+3)u\\
    &\quad+ (q^3+4q^2+6q+3)u^2 +u^3\nonumber\\
    &= 1+ c_{(1)}(4)s_{(1)}(q/u) + c_{(2)}(4)s_{(2)}(q/u) + c_{(3)}(4)s_{(3)}(q/u) + c_{(4)}(4)s_{(4)}(q/u)\\ 
    &\quad+ c_{(5)}(4)s_{(5)}(q/u) + c_{(6)}(4)s_{(6)}(q/u) + c_{(1,1)}(4)s_{(1,1)}(q/u) + c_{(2,1)}(4)s_{(2,1)}(q/u)\nonumber\\ 
    &\quad + c_{(3,1)}(4)s_{(3,1)}(q/u) + c_{(4,1)}(4)s_{(4,1)}(q/u) + c_{(1,1,1)}(4)s_{(1,1,1)}(q/u).\nonumber
\end{align*}
\end{example}

For the remainder of this section, we give a few special cases of Theorem~\ref{thm:improved-coefficient-formula}.
First, we easily recover the following result, at $b=0$, which was already known via the Hilbert series of the classical coinvariant ring $\Hilb(R_n^{(1,0)};q) = [n]_q!$ \cite{Artin}. 

\begin{corollary}\label{cor:coefficient-b=0}
    For $n \geq 1$ and $a \geq 0$, we have
    \begin{equation}
        c_{(a)}(n) = \langle q^a \rangle [n]_q!.
    \end{equation}
\end{corollary}

\begin{proof}
    At $a=0$, we have that $c_\varnothing(n) = 1$.
    For $a \geq 1$, this follows from Theorem~\ref{thm:improved-coefficient-formula} at $b=0$ because then the ordered set partitions are all permutations.
\end{proof}

Next, we give formulas at small values of $a$.

\begin{corollary}\label{cor:KR-comparison} Fix $n \geq 1$. For $0 \leq b \leq n-2$, we have
        \begin{equation}
            c_{(1^{b+1})}(n) = \binom{n-1}{b+1},
        \end{equation}
        and
\begin{equation}
            c_{(2,1^{b})}(n) =
            \binom{n+1}{b+2}\frac{(b+1)(n-b-2)}{n}.
        \end{equation}
\end{corollary}

\begin{proof}
By Theorem~\ref{thm:improved-coefficient-formula}, the coefficient $c_{(a,1^b)}(n)$ counts type I pairs $(B,A)$ with $|B| =b$ and $\sum_i a_i = a$.
Let $j$ be the largest index where $a_j > 0$. Then we have $a_i = 0$ for all $i >j$, $B \subseteq \{ 2,\ldots,j-1\}$, and $\kappa_j = j-1-b$ by equation~\eqref{eq:kappa}.

Observe that a type I ordered set partition whose single inversion occurs at the top index (that is, $a_j=1$ and $a_i = 0 $ for all $i \neq j$) corresponds to a $(b+1)$-element subset $C := B \cup \{j\}$ of $\{2,\ldots,n\}$. 
That is, given such a $C$, the inversion is at index $j = \max C$, $B$ is obtained via $B = C \setminus \{j\}$, and the inversion sequence $A$ is admissible because $\kappa_j = j-1-b \geq 1$ since $j \geq b+2$.

In the case of $a=1$, every such ordered set partition has its single inversion occur at position $j$, so
\begin{equation}
    c_{(1^{b+1})}(n) = \binom{n-1}{b+1}.
\end{equation}

In the case of $a=2$, we consider two subcases depending on the value $a_j \in \{1,2\}$.

If $a_j=2$, the ordered set partition again corresponds to a subset $C$ as before, but for the inversion sequence to be admissible requires $\kappa_j = j-1-b \geq 2$, i.e., $\max C \geq b+3$. 
This only excludes the subset $\{2,\ldots,b+2\}$; hence the enumeration for this case is $\binom{n-1}{b+1} - 1$.

If $a_j=1$, the second inversion is $a_{j'}=1$ for one index $j' < j$, which is admissible as part of a valid inversion sequence if and only if $\kappa_{j'} \geq 1$, that is, $\{2,\ldots,j'\} \not\subseteq B$. 
There are $\binom{j-2}{b}$ choices of $B \subseteq \{2,\ldots, j-1\}$ and $j-2$ candidate positions $j' \in \{2,\ldots, j-1\}$, the inadmissible ones being those with $\{2,\ldots,j'\} \subseteq B$. 
Summing the number of inadmissible positions over all $B$ gives
\begin{equation}
    \sum_B \# \{ 2 \leq j' \leq j-1 : \{2,\ldots,j'\} \subseteq B\} = \sum_{j'=2}^{j-1} \binom{j-1-j'}{b-j'+1} = \binom{j-2}{b-1}.
\end{equation}
Hence, for fixed $j$, the $a_j = 1$ subcase contributes $(j-2)\binom{j-2}{b} - \binom{j-2}{b-1}$.
Summing over $j$, using the hockey-stick identity, we obtain
\begin{equation}
    \sum_{j=b+2}^n \left( (j-2) \binom{j-2}{b} - \binom{j-2}{b-1} \right) = (b+1) \binom{n-1}{b+2} + b\binom{n-1}{b+1} - \binom{n-1}{b} + 1.
\end{equation}

Summing both subcases, using Pascal's identity, gives
\begin{equation}
    \begin{aligned}
        c_{(2,1^b)}(n) &= (b+1)\binom{n-1}{b+2} + (b+1)\binom{n-1}{b+1} - \binom{n-1}{b}\\
        &= \binom{n+1}{b+2}\frac{(b+1)(n-b-2)}{n}.
    \end{aligned}
\end{equation}
\end{proof}

\begin{remark}
    In \cite{Lentfer-Supersymmetry}, there is also a version of diagonal supersymmetry for the Frobenius series of $R_n^{(1,1)}$ in terms of universal series coefficients $c_{\lambda,\mu}$ where $\mu \vdash n$. 
    We can recover the coefficients $c_\lambda(n)$ via $c_\lambda(n) = \sum_{\mu \vdash n} c_{\lambda,\mu} f^\mu$, where $f^\lambda$ is the number of standard Young tableaux of shape $\mu$. 
    Then another proof of Corollary~\ref{cor:KR-comparison} can be given using those coefficients $c_{\lambda,\mu}$ determined in \cite[Proposition 4.3]{Lentfer-Supersymmetry}.
\end{remark}

We conclude this section with an enumerative consequence of Theorem~\ref{thm:improved-coefficient-formula}.
Let $\Fub(n) := \#\OSP(n)$ denote the $n$th Fubini number \cite[\href{https://oeis.org/A000670}{A000670}]{OEIS}.

\begin{corollary}\label{cor:sum-of-coeffs}
    For $n \geq 1$, the number of type I ordered set partitions of $\{1,\ldots, n\}$ is $\frac{1}{2} \left( \Fub(n) -1 \right)$.
    Hence
    \begin{equation}
        \sum_{\lambda \in P(1,1,n)} c_\lambda (n) = \frac{\Fub(n) +1}{2}.
    \end{equation}
\end{corollary}

\begin{proof}
    The involution $\phi$ restricts to a bijection between the type I and the type II ordered set partitions of $\{1,\ldots,n\}$, and its unique fixed point is $\{1\}/\{2\}/ \cdots/\{n\}$, hence there are $\frac{1}{2} \left( \Fub(n) -1 \right)$ type I ordered set partitions.
    Then the second statement follows by summing equation~\eqref{eq:second-formula-coeffs} over all $a \geq 1$ and $b \geq 0$, along with the coefficient $c_\varnothing(n) = 1$.
\end{proof}

\section{The Hilbert series specialization and the conjecture of Sagan--Swanson}\label{sec:SaganSwanson}
In this section, we prove two of our main results, a new formula for the Hilbert series of $R_n^{(1,1)}$ specialized at $u=-q^2$, and the conjecture of Sagan--Swanson.
Let us outline the strategy. 
Lemma~\ref{lem:c_{(a,1^b)}} expresses $c_{(a,1^b)}(n)$ as a signed enumeration of ordered set partitions, and in Section~\ref{sec:main}, the sign-reversing involution $\phi$ of Sagan--Swanson cancels almost all of it: the elements that survive cancellation are the type I ordered set partitions of Theorem~\ref{thm:improved-coefficient-formula}. 
The specialization $u=-q^2$ introduces a second layer of signs, as each coefficient of $\Hilb(R_n^{(1,1)};q;u)|_{u=-q^2}$ is an alternating sum of the coefficients $c_{(a,1^b)}(n)$ along $\ell = a+2b$. 
Accordingly, in Proposition~\ref{prop:sylvie} we construct a second sign-reversing involution $\psi$, defined on the type I ordered set partitions that survived $\phi$.
The fixed points of $\psi$ are enumerated by $\binom{n}{\ell} -1$, which we now establish.

\begin{proposition}\label{prop:sylvie} For $1 \leq \ell \leq n$,
\begin{equation}\label{eq:to-show1}
        \sum_{b=0}^{\lfloor \frac{\ell-1}{2} \rfloor}  (-1)^{b} \#\{w \in \OSP(n,n-b)\, | \, \inv(w) = \ell-2b, \ w \text{ is type I }  \} = \binom{n}{\ell}-1,
    \end{equation}
    and for $\ell > n$, the same sum is $0$.
\end{proposition}

\begin{proof}
Observe that $n-b$ is the number of blocks in each ordered set partition, so the sign $(-1)^b$ equals $(-1)^{n-\blocks(w)}$.

Given an ordered set partition $w=(B,A)$ in $\mathcal T_{n,\ell}$, since it is of type I, it has an associated index $j$. Furthermore, we say that such $w$ is of 
\begin{itemize}
    \item \textbf{type IA} if there exist $k,\hat{\jmath}$ satisfying $k<\hat{\jmath} \leq j$, such that $a_k>0$, $a_{\hat{\jmath}}>1$, $a_i=0$ for all $k<i<\hat{\jmath}$, $a_i\in\{0,1\}$ for all $i>\hat{\jmath}$, and $i\notin B$ for $i\geq k$;
    \item \textbf{type IB} if there exist $k,\hat{\jmath}$ satisfying $k<\hat{\jmath} \leq j$, such that $k\in B$, $a_{\hat{\jmath}}>0$, $a_i=0$ for all $k<i<\hat{\jmath}$, $a_i\in \{0,1\}$ for all $i>\hat{\jmath}$, and $i\notin B$ for all $i>k$;
    \item \textbf{type IC} if $\inv(w)=\ell$ (i.e.,  $B=\emptyset$), there is at most one $\hat{\jmath} \leq j$ such that $a_{\hat{\jmath}}>1$, and if it exists, then $a_1=\cdots =a_{\hat{\jmath}-1}=0$.
\end{itemize}
One may check that any $w \in \mathcal T_{n,\ell}$ is exactly one of type IA, type IB, or type IC.

We define a new involution $\psi$ on $\mathcal T_{n,\ell}$.
\begin{itemize}
\item If $w$ is of type IA, then $\psi(w)=(\tilde B,\tilde A)$ where $\tilde B=B\cup\{k\}$ and $\tilde A$ is given by $\tilde a_i=a_i$ for all $i\neq \hat{\jmath},k$ and $\tilde a_k=a_k-1$ and $\tilde a_{\hat{\jmath}}=a_{\hat{\jmath}}-1$. Then $\psi(w)$ is of type IB.
\item If $w$ is of type IB, then $\psi(w)=(\tilde B,\tilde A)$ where $\tilde B=B\backslash\{k\}$ and $\tilde A$ is given by $\tilde a_i=a_i$ for all $i\neq \hat{\jmath},k$ and $\tilde a_k=a_k+1$ and $\tilde a_{\hat{\jmath}}=a_{\hat{\jmath}}+1$. Then $\psi(w)$ is of type IA.
\item If $w$ is of type IC, then it is a fixed point.
\end{itemize}
One may check that $\psi$ is well defined, that is, that $\psi(w) \in \mathcal T_{n,\ell}$.
The involution $\psi$ is sign-reversing (except at the fixed points) using the same definition of sign as before, $(-1)^{n-\blocks(w)}$.

Let $f_{n,\ell}$ be the number of fixed points of $\psi$, that is, the number of ordered set partitions of type IC in $\mathcal T_{n,\ell}$. 
Note that they are all of the same sign, since they have $n$ blocks. Thus the left hand side of equation~(\ref{eq:to-show1}) simplifies to 
\begin{equation}
    \#\{ w \in \OSP(n,n) \,|\, \text{$w$ is type IC}\}.
\end{equation}
If the ordered set partition $(\emptyset, A)$ is such that $a_i \in \{0,1\}$ for all $i$, then the number of them is $\binom{n-1}{\ell}$, since $a_1$ must be 0.  
If the ordered set partition $(\emptyset, A)$ is such that there exists a $j$ with $a_j>1$, then $(\emptyset, (a_2,\ldots, a_{j-1},a_j-1,a_{j+1},\ldots ,a_n))$ is an ordered set partition of type IC in $\mathcal T_{n-1,\ell-1}$.
This map is a bijection onto the type IC elements of $\mathcal T_{n-1,\ell-1}$, which are enumerated by $f_{n-1,\ell-1}$.
Thus for $n > \ell$, we have
\begin{equation}
    f_{n,\ell}=\binom{n-1}{\ell}+f_{n-1,\ell-1}.
\end{equation}
If $n\leq \ell$, it is easy to see that $f_{n,\ell}=0$.
By induction, with base cases $f_{r,r} = 0$ and $f_{r,0} = 0$ for all $r$, we conclude that
\begin{equation}
f_{n,\ell}=\binom{n}{\ell}-1.
\end{equation}
\end{proof}

Now we are ready to prove one of our main theorems.

\begin{proof}[Proof of Theorem~\ref{thm:main-theorem}]
    By Proposition~\ref{prop:DSS-Hilbert} and Corollary~\ref{cor:schur_specialization_-q^m=u}, we have that
    \begin{equation}
    \begin{aligned}
        \Hilb(R_n^{(1,1)};q;u)|_{u=-q^2} &= \sum_{\lambda \in P(1,1,n)} c_\lambda(n) s_\lambda(q/u)|_{u=-q^2}\\
        &= 1 +\sum_{\substack{a \geq 1,\\ b \geq 0}} c_{(a,1^b)}(n)(-1)^b q^{a+2b}(1-q).
    \end{aligned}
    \end{equation}
    Now by Theorem~\ref{thm:improved-coefficient-formula}, this is equal to 
    \begin{align}
        &1 + (1-q)\sum_{\substack{a \geq 1,\\ b \geq 0}} (-1)^b q^{a+2b} \#\{w \in \OSP(n,n-b)\, | \, \inv(w) = a, \ w \text{ is type I}  \}.
    \end{align} 
Extracting the coefficient of $q^\ell$ from the sum, we have
\begin{equation}
\begin{aligned}
    \langle q^\ell \rangle& \sum_{\substack{a \geq 1,\\ b \geq 0}} (-1)^b q^{a+2b} \#\{w \in \OSP(n,n-b)\, | \, \inv(w) = a,\ w \text{ is type I} \}\\
    &= \sum_{\substack{a\geq 1,b \geq 0\\ \text{s.t. } \ell = a+2b}} (-1)^{b} \#\{w \in \OSP(n,n-b)\, | \, \inv(w) = a,\ w \text{ is type I}\}\\
    &= \sum_{b=0}^{\lfloor \frac{\ell-1}{2} \rfloor} (-1)^{b} \#\{w \in \OSP(n,n-b)\, | \, \inv(w) = \ell-2b,\ w \text{ is type I}\},
\end{aligned}
\end{equation}
which by Proposition~\ref{prop:sylvie} equals $\binom{n}{\ell}-1$ for $1 \leq \ell \leq n$ and equals $0$ for $\ell > n$.

Plugging this back in, and noting that the term $\ell = n$ vanishes since $\binom{n}{n}-1 = 0$, we get that
\begin{equation}
    \Hilb(R_n^{(1,1)};q;u)|_{u=-q^2} = 1+ (1-q)\sum_{\ell=1}^{n-1}q^\ell\left( \binom{n}{\ell} -1\right).
\end{equation}
Finally, extract the coefficient of $q^i$, which is $\binom{n}{i} - \binom{n}{i-1}$, proving the theorem.
\end{proof}

Now we can prove the conjecture of Sagan--Swanson.

\begin{proof}[Proof of Theorem~\ref{thm:sagan-swanson-conjecture}]
    By Theorem~\ref{thm:main-theorem} and Theorem~\ref{thm:hilbert}, we have that
    \begin{equation}
        \left(\sum_{k=0}^n (-q^2)^{n-k}[k]_q!S[n,k]\right)-1 = \sum_{i=1}^n d_i q^i, 
    \end{equation}
where $d_i =\binom{n}{i} - \binom{n}{i-1}$.
Since 
\begin{equation}
    \binom{n}{i} - \binom{n}{i-1} = -\left( \binom{n}{n-i+1} - \binom{n}{n-i}\right),
\end{equation}
it follows that $d_i = -d_{n-i+1}$ for all $i \in \{1,\ldots,n\}$, so they are palindromic ignoring signs.
The coefficients $d_1, \ldots, d_{\lfloor\frac{n}{2} \rfloor}$ are all positive and the coefficients $d_{n -\lfloor\frac{n}{2} \rfloor+1}, \ldots, d_{n}$ are all negative. 
If and only if $n$ is odd, then there is a middle coefficient $d_{\lfloor\frac{n}{2} \rfloor + 1} = 0$.
This completes the proof.
\end{proof}

\section{Specializations at \texorpdfstring{$u = -q^m$}{u=-qm}}\label{sec:generating-function-hub}

The goal of this section is to study the Hilbert series of $R_n^{(1,1)}$ at the more general specialization $u=-q^m$. 
The main result of this section is a new formula for that specialization of the Hilbert series.

Define the generating function
\begin{equation}
    Q_n(q,u) := \sum_{\substack{a\geq 1,\\ b\geq 0}} c_{(a,1^b)}(n) q^a u^b = \sum_{\substack{w \in \OSP(n),\\ w \text{ is type I}}} q^{\inv(w)} u^{n- \blocks(w)},
\end{equation}
where the second equality is Theorem~\ref{thm:improved-coefficient-formula}. 
By Lemma~\ref{lem:schur-1-1}, $s_{(a,1^b)}(q/u) = \frac{q+u}{q}q^a u^b$, so we may write
\begin{equation}
    \Hilb(R_n^{(1,1)};q;u) = 1 +  \frac{q+u}{q}Q_n(q,u).
\end{equation}
For integers $n \geq 0$ and $m\geq 0$, set
\begin{equation}\label{eq:H-definition}
    \mathcal H_n^{(m)} := \Hilb(R_n^{(1,1)};q;u)\big|_{u=-q^m},
\end{equation}
where $R_0^{(1,1)} := \C$, so that $\mathcal H_0^{(m)} = 1$.
For $m \geq 1$, we may write
\begin{equation}
    \mathcal H_n^{(m)} = 1 +  (1-q^{m-1})Q_n(q,-q^m).
\end{equation}
Define, for $m \geq 1$,
\begin{equation}\label{eq:K-definition}
    \mathcal K_n^{(m)} := Q_n(q,-q^m). 
   \end{equation}
Then for $m \geq 2$,
   \begin{equation}\label{eq:K-fraction}
   \mathcal K_n^{(m)} =  \frac{\mathcal H_n^{(m)} - 1}{1-q^{m-1}} = \sum_{\substack{a \geq 1,\\ b \geq 0}} c_{(a,1^b)}(n)\,(-1)^b q^{a+mb}.
\end{equation}

Our next goal is to establish recurrences and a closed-form expression for 
\begin{equation}\label{eq:for-lemma-r1}
    \mathcal H_n^{(m)}=\sum_{k=0}^n(-1)^{n-k}q^{m(n-k)}S^o[n,k].
\end{equation}
First we will prove two lemmas.
For $n \geq 0$ and $m \geq 1$, we define an auxiliary sum
    \begin{equation}\label{eq:def-A}
        \mathcal A_n^{(m)} := \sum_{k=0}^n(-1)^{n-k}q^{m(n-k)}[k]_q  S^o[n,k].
    \end{equation}

\begin{lemma}\label{lem:recurrence1}
For $n \geq 0$ and $m \geq 1$,
\begin{equation}
    \mathcal H_{n+1}^{(m)} = (1-q^m)\mathcal A_n^{(m)} + q^n \mathcal H_n^{(m-1)}.
\end{equation}
\end{lemma}

\begin{proof}
    Apply $S^o[n+1,k] = [k]_q \left( S^o[n,k-1] + S^o[n,k]\right)$ to equation~\eqref{eq:for-lemma-r1}, yielding
    \begin{equation}
    \begin{aligned}
        \mathcal H_{n+1}^{(m)} &= \sum_{k=0}^{n+1} (-1)^{n+1-k} q^{m(n+1-k)} [k]_q \left( S^o[n,k-1] + S^o[n,k]\right)\\
        &= \mathcal A_n^{(m)} + \sum_{k=0}^n(-1)^{n-k}q^{m(n-k)+k}S^o[n,k] -q^m \mathcal A_n^{(m)}.
    \end{aligned}
    \end{equation}
   We find that
    \begin{equation}
        \sum_{k=0}^n(-1)^{n-k}q^{m(n-k)+k}S^o[n,k] = q^n\sum_{k=0}^n(-1)^{n-k} q^{(m-1)(n-k)}S^o[n,k] = q^n\mathcal H_n^{(m-1)}.
    \end{equation}
    Combining these completes the proof.
\end{proof}

\begin{lemma}\label{lem:recurrence2}
    For $n\geq0$ and $m\geq1$,
\begin{equation}
    \mathcal A_n^{(m)} = \frac{\mathcal H_n^{(m)} - q^n\,\mathcal H_n^{(m-1)}}{1-q}.
\end{equation}
\end{lemma}

\begin{proof}
    Since $[k]_q=(1-q^k)/(1-q)$, this follows from equations~(\ref{eq:for-lemma-r1},~\ref{eq:def-A}) and a bit of algebra.
\end{proof}

We now prove a recurrence for the Hilbert series specializations $\mathcal H_n^{(m)}$.

\begin{proposition}\label{prop:H-recurrence}
    For all $m \geq 1$,
    \begin{equation}
        \mathcal H_n^{(m)} = [m]_q \mathcal H_{n-1}^{(m)} - q^n [m-1]_q \mathcal H_{n-1}^{(m-1)},
    \end{equation}
    for $n \geq 1$, and $\mathcal H_0^{(m)} = 1$.
    The base case is $\mathcal H_n^{(1)} = 1$ for all $n$.
\end{proposition}

\begin{proof}
Plug Lemma~\ref{lem:recurrence2} at $n \mapsto n-1$ into Lemma~\ref{lem:recurrence1}, and simplify with the identities $[m]_q = \frac{1-q^m}{1-q}$ and $[m]_q-1 = q[m-1]_q$.    
By definition, $\mathcal H_0^{(m)} = 1$; then use the recurrence relation to compute that $\mathcal H_n^{(1)} = \mathcal H_{n-1}^{(1)} = \cdots = \mathcal H_0^{(1)} = 1$. 
Note that this base case of $\mathcal H_n^{(1)} = 1$ was previously shown by Sagan--Swanson \cite[Theorem 5.5]{SaganSwanson2024}.
\end{proof}

Define (a special case of) the \textbf{$q$-Pochhammer symbol} by
\begin{equation}
    (q;q)_r = \prod_{i=1}^r (1-q^i).
\end{equation}
By induction on $r \geq 0$, one can show that 
\begin{equation}\label{eq:pochhammer}
        \sum_{i=0}^{r} \frac{q^i}{(q;q)_i} = \frac{1}{(q;q)_r}.
\end{equation}

Now we can prove a closed-form formula for $\mathcal H_n^{(m)}$ and $\mathcal K_n^{(m)}$.

\begin{theorem}\label{thm:H-K-closed-form}
    For all $m \geq 1$ and $n \geq 0$,
    \begin{equation}\label{eq:H-closed-form}
        \mathcal H_n^{(m)} = (q;q)_{m-1}\sum_{j=0}^{m-1}  \frac{q^{j(n+1)}}{(q;q)_{j}}[m-j]_q^n.
    \end{equation}
    Furthermore, if $m \geq 2$ and $n \geq 0$,
    \begin{equation}\label{eq:K-closed-form}
        \mathcal K_n^{(m)} = (q;q)_{m-2} \sum_{j=0}^{m-2} \frac{q^{j(n+1)}}{(q;q)_j}[m-j]_q^n -[n+1]_{q^{m-1}}.
    \end{equation}
\end{theorem}

\begin{proof}
    We will show equation~\eqref{eq:H-closed-form} by induction on $m$.
    For concision, we write
    \begin{equation}
        \beta_{m,j} := q^{j}[m-j]_q, \quad \text{ and } \quad \alpha_{m,j} = q^{j} \frac{(q;q)_{m-1}}{(q;q)_{j}}.
    \end{equation}
    Define
    \begin{equation}
        \mathcal G_n^{(m)} := \sum_{j=0}^{m-1} \alpha_{m,j}\beta_{m,j}^n,
    \end{equation}
    which is the right-hand side of equation~\eqref{eq:H-closed-form}.

    It is easy to verify the following identities:
    \begin{equation}\label{eq:first-identity}
        \beta_{m,j} - [m]_q = - [j]_q \quad \text{ and } \quad \beta_{m,j} = q\beta_{m-1,j-1},
    \end{equation}
    and for $1 \leq j \leq m-1$,
    \begin{equation}\label{eq:second-identity}
        \alpha_{m,j}[j]_q = q[m-1]_q\alpha_{m-1,j-1}.
    \end{equation}

    Now we proceed with our induction on $m$. 
    For the base case of $m=1$, $\mathcal G_n^{(1)} = 1 = \mathcal H_n^{(1)}$.
    For the inductive step, assume that $\mathcal G_n^{(m-1)} = \mathcal H_n^{(m-1)}$ for all $n \geq 0$. 
    In the case that $n=0$, we have
    \begin{equation}
        \mathcal G_0^{(m)} = \sum_{j=0}^{m-1} \alpha_{m,j} = (q;q)_{m-1} \sum_{i=0}^{m-1} \frac{q^i}{(q;q)_i} = 1,
    \end{equation}
    by equation~\eqref{eq:pochhammer}.
    Next, by equations~(\ref{eq:first-identity}, \ref{eq:second-identity}) and $\beta_{m,j}^{n-1} = q^{n-1} \beta_{m-1,j-1}^{n-1}$, we write
    \begin{equation}
    \begin{aligned}
        \mathcal G_n^{(m)} - [m]_q \mathcal G_{n-1}^{(m)} &=\sum_{j=0}^{m-1} (\beta_{m,j} - [m]_q) \alpha_{m,j}\beta_{m,j}^{n-1}\\ 
        &=  -\sum_{j=0}^{m-1} [j]_q \alpha_{m,j}q^{n-1} \beta_{m-1,j-1}^{n-1}\\
        &= -q^n [m-1]_q \sum_{j=0}^{m-2}  \alpha_{m-1,j} \beta_{m-1,j}^{n-1}\\
        &= -q^n [m-1]_q \mathcal G_{n-1}^{(m-1)}.
    \end{aligned}
    \end{equation}
    Thus $\mathcal G_n^{(m)}$ satisfies the same recurrence relation and initial conditions as $\mathcal H_n^{(m)}$ (Proposition~\ref{prop:H-recurrence}), so $\mathcal G_n^{(m)} = \mathcal H_n^{(m)}$.

    Now  the second claim follows easily using
    \begin{equation}
        \mathcal K_n^{(m)} = \frac{\mathcal H_n^{(m)} - 1}{1-q^{m-1}}.
    \end{equation}
\end{proof}

At $m=2$, we immediately deduce the following (this recovers Theorem~\ref{thm:main-theorem}).
\begin{corollary}\label{cor:2^n-n-1}
For $n \geq 1$,
\begin{equation}
\mathcal H_n^{(2)} = (1-q)[2]_q^n + q^{n+1}, \quad \text{ and } \quad
\mathcal K_n^{(2)} = [2]_q^n - [n+1]_q.
\end{equation}
\end{corollary}

At the $q=1$ specialization, we deduce the following.
\begin{corollary}
For $n \geq 1$ and $m \geq 2$,
\begin{equation}
\mathcal H_n^{(m)}|_{q=1} = 1, \quad \text{ and } \quad
\mathcal K_n^{(m)}|_{q=1} = 2^n - n- 1.
\end{equation}
\end{corollary}

\section{Coefficients of \texorpdfstring{$\mathcal K_n^{(m)}$}{Kn(m)}}\label{sec:coefficients}

Recall from Theorem~\ref{thm:H-K-closed-form} that $\mathcal K_n^{(m)}$ contains the essential content of the Hilbert series specialization $\mathcal H_n^{(m)}$.
Since $\mathcal K_n^{(m)}$ is a polynomial in $q$, a natural question is to study the coefficients $u_\ell(n,m) := \langle q^\ell \rangle \mathcal K_n^{(m)}$.
In this section, we determine the support of $\mathcal K_n^{(m)}$ and compute its coefficients at extreme low and high degrees.
We first bound the maximum degree for which $\mathcal K_n^{(m)}$ is supported.

\begin{proposition}\label{prop:degbound}
    For all $n\geq 0$ and $m\geq 1$ we have $\deg_q \mathcal H_n^{(m)}\le(m-1)n$. 
    For all $n\geq 0$ and $m\geq 2$ we have $\deg_q \mathcal K_n^{(m)}\le(m-1)(n-1)$. 
    Consequently, for $m \geq 2$, $u_\ell(n,m)=0$ whenever $\ell>(m-1)(n-1)$.
\end{proposition}

\begin{proof}
    We prove $\deg_q\mathcal H_n^{(m)}\leq(m-1)n$ for all $m\geq1$ simultaneously, by induction on $n$.
    For the base case of $n=0$, $\mathcal H_0^{(m)} = 1$ has degree $0$.

    Assume the bound holds at $n$ for every $m \geq 1$. Fix $m$.
    If $m = 1$, then $\mathcal H_{n+1}^{(1)} = 1$ has degree $0 = (1-1)(n+1)$. If $m \geq 2$, the inductive hypothesis gives $\deg_q\mathcal H_n^{(m)}\leq(m-1)n$ and $\deg_q\mathcal H_n^{(m-1)}\leq(m-2)n$, hence we write
    \begin{equation}
        \deg_q \left(\mathcal H_n^{(m)} - q^n \mathcal H_n^{(m-1)} \right) \leq \max \{ (m-1)n, n+ (m-2)n\} = (m-1)n.
    \end{equation}
    By Lemma~\ref{lem:recurrence2}, $\mathcal H_n^{(m)} - q^n \mathcal H_n^{(m-1)}$ is divisible as a polynomial by $1-q$, and $\deg_q \mathcal A_n^{(m)}\leq(m-1)n-1$.
    Then by Lemma~\ref{lem:recurrence1},
    \begin{equation}
    \begin{aligned}
        \deg_q\mathcal H_{n+1}^{(m)} &\leq  \max \{ m + \deg_q \mathcal A_n^{(m)}, n + \deg_q \mathcal H_n^{(m-1)} \}\\
        &\leq \max \{ (m-1)n + (m-1), (m-1)n\} = (m-1)(n+1).
    \end{aligned}
    \end{equation}
    This completes the induction.

    Finally, by equation~\eqref{eq:K-fraction} we have that $\mathcal K_n^{(m)} = (\mathcal H_n^{(m)}-1)/(1-q^{m-1})$, which is a polynomial, so 
    \begin{equation}
        \deg_q \mathcal K_n^{(m)}=\deg_q(\mathcal H_n^{(m)}-1)-(m-1)\leq(m-1)n-(m-1)=(m-1)(n-1).
    \end{equation}
    Hence $u_\ell(n,m)=\langle q^\ell\rangle \mathcal K_n^{(m)}=0$ for $\ell>(m-1)(n-1)$.
\end{proof}

We now compute the low-degree coefficients in two ways.

\begin{proposition}\label{prop:u_ell-values}
    For $m\geq 1$ we have $u_0(n,m)=0$ and $u_\ell(n,m)=\langle q^\ell\rangle[n]_q!$ for $1\leq\ell<m$.
\end{proposition}

\begin{proof}
    By Theorem~\ref{thm:improved-coefficient-formula}, $c_{(a,1^b)}(n)$ counts type I ordered set partitions $w$ in $\OSP(n,n-b)$ with $\inv(w) = a$. Since $n-b = \blocks(w)$, we write
    \begin{equation}
        \mathcal K_n^{(m)} = \sum_{\substack{w \in \OSP(n),\\
        w \text { is type I}}} (-1)^{n-\blocks(w)} q^{\inv(w) + m(n-\blocks(w))}.
    \end{equation}
    Thus
    \begin{equation}\label{eq:contributing}
        u_\ell(n,m) = \sum_{\substack{w \in \OSP(n),\\
        w \text { is type I}, \\ \ell = \inv(w)+ m(n-\blocks(w))}} (-1)^{n-\blocks(w)}.
    \end{equation}
    Every type I ordered set partition has $\inv(w)\geq 1$; so $u_0(n,m)=0$ for $m\geq 1$.

    Fix $1 \leq \ell < m$. 
    Suppose that $w$ contributes to the sum.
    If $n -\blocks(w) \geq 1$, then $\ell = \inv(w) + m(n-\blocks(w)) \geq m > \ell$, a contradiction; hence $\blocks(w) = n$.
    These are exactly the ordered set partitions of all singleton blocks, which can be identified with the set of permutations of $n$, $\mathfrak{S}_n$, and they all contribute with a sign of $+1$, exactly when $\inv(w) = \ell$.
    Hence
    \begin{equation}
        u_\ell(n,m) = \#\{ w\in \mathfrak{S}_n \, | \, \inv(w) =\ell, w \text{ is type I}\}.
    \end{equation}
    Now a permutation is a type I ordered set partition if and only if it has at least one inversion. Hence for $1 \leq \ell < m$, we have
    \begin{equation}
        u_\ell(n,m) = \langle q^\ell \rangle [n]_q!.
    \end{equation}
\end{proof}

We establish next another formula for the low-degree coefficients for a different bound on $\ell$.
\begin{proposition}\label{prop:low-degree-coefficients}
    For $1 \leq \ell \leq n$ and $m \geq 2$,
\begin{equation}
    u_\ell(n,m)
    = \langle q^\ell \rangle \left( (q;q)_{m-2} [m]_q^n \right)
    - \begin{cases}
        1 & \text{if } m-1 \mid \ell, \\
        0 & \text{otherwise.}
    \end{cases}
\end{equation}
In particular, $u_1(n,m) = n-1$ for all $m \geq 2$.
\end{proposition}

\begin{proof}
This is straightforward as for $\ell\le n$, only the summand $j=0$
in \eqref{eq:K-closed-form} contributes.
\end{proof}

At the opposite end of the support, we compute the leading coefficient.
\begin{proposition}\label{prop:leading-coefficient}
    For $n \geq 1$ and $m \geq 1$,
    \begin{equation}
        \langle q^{(m-1)n} \rangle \mathcal H_n^{(m)} = (-1)^{m-1} \binom{n-1}{m-1},
    \end{equation}
    hence for $m \geq 2$,
    \begin{equation}
        \langle q^{(m-1)(n-1)} \rangle \mathcal K_n^{(m)} = (-1)^{m} \binom{n-1}{m-1}.
    \end{equation}
    In particular, for $m \geq 2$, $\deg_q \mathcal K_n^{(m)} = (m-1)(n-1)$ exactly when $m \leq n$.
\end{proposition}

\begin{proof}
Write $L(n,m) := \langle q^{(m-1)n} \rangle \mathcal H_n^{(m)}$.
Since $[m]_q$ is monic of degree $m-1$ and $\deg_q \mathcal H_{n-1}^{(m)} \leq (m-1)(n-1)$ by Proposition~\ref{prop:degbound}, the product $[m]_q\mathcal H_{n-1}^{(m)}$ has degree at most $(m-1)n$, and $\langle q^{(m-1)n} \rangle [m]_q\mathcal H_{n-1}^{(m)} = L(n-1,m)$.
Similarly, $q^n[m-1]_q$ is monic of degree $n+m-2$ and $\deg_q \mathcal H_{n-1}^{(m-1)} \leq (m-2)(n-1)$, so $q^n[m-1]_q \mathcal H_{n-1}^{(m-1)}$ has degree at most $(m-1)n$, and $\langle q^{(m-1)n} \rangle q^n[m-1]_q\mathcal H_{n-1}^{(m-1)} = L(n-1,m-1)$.
Extracting the coefficient of $q^{(m-1)n}$ from the recurrence of Proposition~\ref{prop:H-recurrence} gives
    \begin{equation}
        L(n,m) = L(n-1,m) - L(n-1,m-1).
    \end{equation}
With base cases $L(n,1) = 1$ and $L(1,m) = 0$ for $m \geq 2$, induction on $n$ using Pascal's rule gives $L(n,m) = (-1)^{m-1} \binom{n-1}{m-1}$.

Finally, $(1-q^{m-1})\mathcal K_n^{(m)} = \mathcal H_n^{(m)} -1$ 
together with the bound $\deg_q \mathcal K_n^{(m)} \leq (m-1)(n-1)$ gives $\langle q^{(m-1)(n-1)} \rangle \mathcal K_n^{(m)}  = - L(n,m) = (-1)^m \binom{n-1}{m-1}$, which is nonzero if and only if $m -1 \leq n-1$. 
Combining it with the same degree bound shows that $\deg_q \mathcal K_n^{(m)} = (m-1)(n-1)$ exactly when $m \leq n$.
\end{proof}

\section{Generating functions \texorpdfstring{$\sum_a c_{(a,1^b)}(n) q^a$}{c(a,1b)(n)}}\label{sec:row-gf}

In this section, we work with a new generating function for the coefficients $c_{(a,1^b)}(n)$; our goal is to prove the following result.

\begin{theorem}\label{thm:unsigned-gf}
    Fix $n \geq 1$. For $0 \leq b \leq n-1$, we have the generating function 
    \begin{equation}
        \sum_{a \geq 1} c_{(a,1^b)}(n)q^a = q^{1- \binom{n-b-1}{2}} \sum_{j=0}^{n-b-1}(-1)^j q^{\binom{j}{2}} \qbinom{n-b}{j}_q  [n-b-j-1]_q^n.
    \end{equation}
\end{theorem}

For $1 \leq k \leq n$, we define
\begin{equation}
    S^I[n,k] := \sum_{a \geq 0}\#\{w \in \OSP(n,k)\, | \, \inv(w) = a, \ w \text{ is type I}  \}\, q^a,
\end{equation}
so that, by Theorem~\ref{thm:improved-coefficient-formula}, $\sum_{a \geq 1} c_{(a,1^b)}(n) q^a = S^I[n,n-b]$.
Recall that 
\begin{equation}
    S^o[n,k]= [k]_q!S[n,k] = \sum_{a \geq 0} \#\{w \in \OSP(n,k)\, | \, \inv(w) = a\}\, q^a,
\end{equation}
for $1 \leq k \leq n$.
We prove the following.

\begin{lemma}\label{lem:SI-recurrence}
For $n\geq k\geq 2$,
\begin{equation}\label{eq:SI-recurrence}
S^I[n,k]=S^I[n-1,k-1]+q[k-1]_q S^o[n-1,k-1],
\end{equation}
with initial conditions $S^I[n,1]=0$ for all $n \geq 1$.
\end{lemma}

\begin{proof}
    Since a single-block ordered set partition has no inversions, no ordered set partition in $\OSP(n,1)$ is type I, so $S^I[n,1]=0$.

    For $n \geq k \geq 2$, we use the encoding $w = (B,A)$.
    Recall that $w \in \OSP(n,k)$ corresponds to a pair with $|B| = n-k$ and $\inv(w) = \sum_{i=1}^n a_i$, and that $w$ is type I exactly when there exists an index $j$ with $a_j > 0$, with $a_i = 0$ for all $i > j$, and with $i \notin B$ for all $i \geq j$.
    On the one hand, if $a_n > 0$ then this index is $j=n$, so $n \notin B$. On the other hand, if $a_n = 0$ then the index satisfies $j <n$, so again $n \notin B$. In either case $B \subseteq \{2,\ldots,n-1\}$, hence $(B,(a_1,\ldots,a_{n-1}))$ is an ordered set partition in $\OSP(n-1,k-1)$. Consider again the two cases.
    \begin{itemize}
        \item If $a_n = 0$, then $(B,(a_1,\ldots,a_{n-1}))$ is type I, with the same witnessing index $j < n$ and same number of inversions. This is a bijection onto the type I ordered set partitions in $\OSP(n-1,k-1)$, so it contributes $S^I[n-1,k-1]$.
        \item If $a_n > 0$, then $0 < a_n < n - |B \cap \{1,\ldots,n\}| = k$, so $a_n \in \{1,\ldots, k-1\}$, and $(B,(a_1,\ldots,a_{n-1}))$ is an arbitrary ordered set partition in $\OSP(n-1,k-1)$ with $\inv(w)-a_n$ inversions.
        Conversely, any such pair together with any value $a_n \in \{1,\ldots, k-1\}$ yields a type I ordered set partition in $\OSP(n,k)$ with witness $j=n$.
        Summing over the value of $a_n$ contributes $\sum_{j=1}^{k-1} q^j S^o[n-1,k-1]= q[k-1]_qS^o[n-1,k-1]$.
    \end{itemize}
    Adding together the two cases completes the proof.
\end{proof}

Now we prove the main theorem of this section. 

\begin{proof}[Proof of Theorem~\ref{thm:unsigned-gf}]
For $1 \leq k \leq n$, define
\begin{equation}\label{eq:Def-F}
    F[n,k] := q^{1- \binom{k-1}{2}} \sum_{j=0}^{k-1}(-1)^j q^{\binom{j}{2}} \qbinom{k}{j}_q  [k-j-1]_q^n.
\end{equation}
Writing $k=n-b$, it suffices by Theorem~\ref{thm:improved-coefficient-formula} to prove that $S^I[n,k] = F[n,k]$. We will induct on $k$.

For the base case of $k=1$, the sum in equation~\eqref{eq:Def-F} has only the $j=0$ term: $[0]_q^n = 0$, so $F[n,1] = 0 = S^I[n,1]$ by Lemma~\ref{lem:SI-recurrence}.

Now let $k \geq 2$, and assume that $S^I[n-1,k-1] = F[n-1,k-1]$. 
By Lemma~\ref{lem:SI-recurrence} and the inductive hypothesis,
\begin{equation}
    S^I[n,k] = F[n-1,k-1] + q[k-1]_q S^o[n-1,k-1],
\end{equation}
so it suffices to prove the polynomial identity
\begin{equation}\label{eq:F-identity}
    F[n,k] = F[n-1,k-1] + q[k-1]_q S^o[n-1,k-1].
\end{equation}

Apply Pascal's identity 
\begin{equation}
    \qbinom{k}{j}_q = q^j \qbinom{k-1}{j}_q + \qbinom{k-1}{j-1}_q
\end{equation}
to $F[n,k]$ and the fact that $[k-j-1]_q = q^{-j}([k-1]_q - [j]_q)$ to write that
\begin{equation}
\begin{aligned}
    F[n,k] &= q^{1-\binom{k-1}{2}}[k-1]_q\sum_{j=0}^{k-1} (-1)^j q^{\binom{j}{2}} \qbinom{k-1}{j}_q [k-j-1]_q^{n-1}\\
    &\quad- q^{1-\binom{k-1}{2}}\sum_{j=0}^{k-1} (-1)^j q^{\binom{j}{2}}[j]_q\qbinom{k-1}{j}_q [k-j-1]_q^{n-1}\\
    &\quad+ q^{1-\binom{k-1}{2}}\sum_{j=0}^{k-1} (-1)^j q^{\binom{j}{2}}\qbinom{k-1}{j-1}_q[k-j-1]_q^n,
\end{aligned}
\end{equation}
which by Lemma~\ref{lem:SO-closed-form} simplifies to
\begin{align}
        F[n,k] &= q[k-1]_q S^o[n-1,k-1]\\
        &\quad + q^{1-\binom{k-1}{2}}\sum_{j=0}^{k-1} (-1)^j q^{\binom{j}{2}}\left(\qbinom{k-1}{j-1}_q [k-j-1]_q^{n} - [j]_q\qbinom{k-1}{j}_q[k-j-1]_q^{n-1}\right).\label{line:second-line}
\end{align}
Thus showing equation~\eqref{eq:F-identity} reduces to showing that line~\eqref{line:second-line} equals $F[n-1,k-1]$.
By applying the identity
\begin{equation}
    [k-j-1]_q \qbinom{k-1}{j-1}_q -[j]_q\qbinom{k-1}{j}_q = -q^{k-j-1}\qbinom{k-1}{j-1}_q,
\end{equation}
line~\eqref{line:second-line} becomes
\begin{equation}
    q^{1-\binom{k-1}{2}}\sum_{j=1}^{k-1} (-1)^j q^{\binom{j}{2}}\left(-q^{k-j-1}\qbinom{k-1}{j-1}_q\right)[k-j-1]_q^{n-1},
\end{equation}
where the $j=0$ term vanished.
By shifting the index of summation from $j \mapsto j+1$, we conclude that line~\eqref{line:second-line} equals
\begin{equation}
    q^{1-\binom{k-1}{2}+k-2}\sum_{j=0}^{k-2} (-1)^j q^{\binom{j+1}{2}-j}\qbinom{k-1}{j}_q[k-j-2]_q^{n-1} = F[n-1,k-1].
\end{equation}
This proves equation~\eqref{eq:F-identity}, completing the induction.
\end{proof}

\section{Future directions}

The conjecture of Sagan--Swanson (Theorem~\ref{thm:sagan-swanson-conjecture}) states that at $m=2$, there are two differently signed regions of the coefficients (where the \textbf{signed regions} are the maximal consecutive runs of coefficients of the same sign, which partition the nonzero coefficients). 
By looking at data for $m \geq 3$, it seems that a generalization of this phenomenon occurs.
As an example, we include data for $\Hilb(R_n^{(1,1)}; q;u) |_{u=-q^m} -1 = \mathcal H_n^{(m)} - 1$ for $2 \leq m \leq 4$ and $2 \leq n \leq 6$.
\begin{example}[$m=2$] 
    \begin{align*}
        \mathcal H_2^{(2)} - 1 &= q - q^2,\\
        \mathcal H_3^{(2)} - 1 &= 2q - 2q^3,\\
        \mathcal H_4^{(2)} - 1 &= 3q + 2q^2 - 2q^3 - 3q^4,\\
        \mathcal H_5^{(2)} - 1 &= 4q + 5q^2 - 5q^4 - 4q^5,\\
        \mathcal H_6^{(2)} - 1 &= 5q + 9q^2 + 5q^3 - 5q^4 - 9q^5 - 5q^6.
    \end{align*}
\end{example}
\begin{example}[$m=3$] 
    \begin{align*}
        \mathcal H_2^{(3)} - 1 &= q - q^3,\\
        \mathcal H_3^{(3)} - 1 &= 2q + 2q^2 - q^3 - 3q^4 - q^5 + q^6,\\
        \mathcal H_4^{(3)} - 1 &= 3q + 5q^2 + 3q^3 - 3q^4 - 8q^5 - 5q^6 + 2q^7 + 3q^8,\\
        \mathcal H_5^{(3)} - 1 &= 4q + 9q^2 + 11q^3 + 5q^4 - 9q^5 - 20q^6 - 16q^7 + 10q^9 + 6q^{10},\\
        \mathcal H_6^{(3)} - 1 &= 5q + 14q^2 + 24q^3 + 25q^4 + 7q^5 - 25q^6 - 50q^7 - 45q^8 - 11q^9 + 21q^{10} + 25q^{11} + 10q^{12}.
    \end{align*}
\end{example}
\begin{example}[$m=4$] 
    \begin{align*}
        \mathcal H_2^{(4)} - 1 &= q - q^4,\\
        \mathcal H_3^{(4)} - 1 &= 2q + 2q^2 + q^3 - 2q^4 - 3q^5 - q^6 + q^8,\\
        \mathcal H_4^{(4)} - 1 &= 3q + 5q^2 + 6q^3 + 2q^4 - 5q^5 - 10q^6 - 9q^7 - q^8 + 5q^9 + 4q^{10} + q^{11} - q^{12},\\
        \mathcal H_5^{(4)} - 1 &= 4q + 9q^2 + 15q^3 + 16q^4 + 7q^5 - 11q^6 - 31q^7 - 36q^8 - 20q^9 + 7q^{10} + 25q^{11}\\ 
        &\quad+ 20q^{12} + 4q^{13} - 5q^{14} - 4q^{15},\\
        \mathcal H_6^{(4)} - 1 &= 5q + 14q^2 + 29q^3 + 44q^4 + 47q^5 + 26q^6 - 24q^7 - 85q^8 - 122q^9 - 105q^{10}\\ 
        &\quad- 31q^{11} + 56q^{12} + 101q^{13} + 80q^{14} + 21q^{15} - 21q^{16} - 25q^{17} - 10q^{18}.
    \end{align*}
\end{example}
Observe that in all of these examples, as $n$ increases, the number of signed regions eventually increases to $m$. 
It would be interesting to further study this pattern, perhaps via $m$ nested sign-reversing involutions.
We have verified the following for all $n,m \leq 13$.
\begin{conjecture}
    Fix $m \geq 2$. Then for all $n \geq 1$, there are $\min(n,m)$ signed regions of the polynomial $\Hilb(R_n^{(1,1)}; q;u) |_{u=-q^m} = \mathcal H_n^{(m)}$.
\end{conjecture}

In this paper, we only worked in type $A$, and used the ground field $\C$.
The superspace coinvariant ring has been studied in type $B$ \cite{SaganSwanson2024, bhattacharya} and for wreath product groups \cite{SaganSwanson-ComplexReflection, RhoadesBhattacharya}; it has also been studied over $\GL_n(\mathbb{F}_q)$ \cite{RhoadesWilson-finite-fields}. 
It would be interesting to see if any of our results can be extended to those settings.

\section*{Acknowledgements}

The authors thank Bruce Sagan and Josh Swanson for helpful conversations. 
This material is partially based upon work supported by the National Science Foundation under Grant No. DMS-1929284 while the second author was in residence at ICERM in Providence, RI, during the Categorification and Computation in Algebraic Combinatorics semester program.
The second author was partially supported by the National Science Foundation Graduate Research Fellowship DGE-2146752.

\bibliographystyle{amsplain}
\bibliography{biblio}

\end{document}